\documentclass[a4paper,twoside,11pt]{amsart}

\usepackage{graphicx}
\usepackage[english, french]{babel}
\usepackage[T1]{fontenc}
\usepackage[ansinew]{inputenc}

\usepackage{amsmath,enumerate,color,nccmath,geometry,bm}
 \numberwithin{equation}{section}
\usepackage{amsthm}
\usepackage{amsfonts}
\usepackage{amssymb}
\usepackage{amsbsy}
\usepackage{graphics}
\usepackage{stmaryrd}
\usepackage{amsfonts}
\usepackage{setspace}
\usepackage{mathtools}
\usepackage{hyperref}

\DeclareMathAlphabet{\eusm}{U}{}{}{}  
\SetMathAlphabet\eusm{normal}{U}{eus}{m}{n}
\SetMathAlphabet\eusm{bold}{U}{eus}{b}{n}

\DeclareMathAlphabet{\eufrak}{U}{}{}{}  
\SetMathAlphabet\eufrak{normal}{U}{euf}{m}{n}
\SetMathAlphabet\eufrak{bold}{U}{euf}{b}{n}

\newtheorem{theorem}{Theorem}[section]

\newtheorem{lemma}[theorem]{Lemma}

\theoremstyle{definition}
\newtheorem{definition}[theorem]{Definition}
\newtheorem{example}[theorem]{Example}

\theoremstyle{remark}
\newtheorem{remark}[theorem]{Remark}
\numberwithin{equation}{section}

\newcommand{\R}{\mathbb{R}}
\newcommand{\C}{\mathbb{C}}
\newcommand{\Z}{\mathbb{Z}}
\newcommand{\N}{\mathbb{N}}
  
\newcommand\Ha{{\mathbb C}^+}       
 
\newcommand\eps{\varepsilon} 

\renewcommand{\Im}{\text{\normalfont Im}} 
\renewcommand{\epsilon}{\varepsilon}  
\newcommand{\supp}{\operatorname{supp}}  

\def\mm{\kern0.2em\rule{0.035em}{0.52em}\kern-.35em\gtrdot \kern-.2em}
\def\mmm{\mathop{\kern0.1em\lower0.1ex\hbox{\rule{0.03em}{0.42em}}\kern-.1em\gtrdot \kern0.02em}}
\def\submm{\mathop{{\kern0.1em\lower0.07ex\hbox{\rule{0.035em}{0.57em}}\kern-.1em\gtrdot \kern0.02em}}}

\makeatletter
\renewenvironment{abstract}[1][\languagename]{%
  \ifx\maketitle\relax
    \ClassWarning{\@classname}{Abstract should precede
      \protect\maketitle\space in AMS document classes; reported}%
  \fi
  \global\setbox\abstractbox=\vtop \bgroup
    \unvbox\abstractbox 
    \medskip
    \normalfont\Small
    \list{}{\labelwidth\z@
      \leftmargin3pc \rightmargin\leftmargin
      \listparindent\normalparindent \itemindent\z@
      \parsep\z@ \@plus\p@
      
    }%
    \csname otherlanguage*\endcsname{#1}\csname captions#1\endcsname
    \item[\hskip\labelsep\scshape\abstractname.]%
}{%
  \endlist\egroup
  \ifx\@setabstract\relax \@setabstracta \fi
}
\makeatother

\begin{document}
\parindent 0pt 
\definecolor{red}{rgb}{0.9, 0.0, 0.0}
\definecolor{magenta}{rgb}{0, 0.9, 0.0}
\pagestyle{empty} 

\setcounter{page}{1}

\pagestyle{plain}

\title{Loewner's differential equation and spidernets}

\author{Sebastian Schlei{\ss}inger}
\thanks{Supported by
the German Research Foundation (DFG), project no. 401281084.}

\date{\today}

\begin{abstract}[english]
We regard a certain type of Loewner's differential equation from a quantum probability point of view and approximate the 
underlying quantum process by the adjacency matrices of growing graphs which 
arise from the comb product of certain spidernets.
\end{abstract}
\selectlanguage{english}


\maketitle

\tableofcontents


\section{Introduction}

The Loewner equation  
\begin{equation}\label{slit0}
\frac{\partial g_t(z)}{\partial t} =  \frac{1}{g_t(z) - U(t)}, \quad g_0(z)=z\in\Ha:=\{z\in\C\,|\, \Im(z)>0\},
\end{equation} where $U:[0,\infty)\to \R$ is continuous, 
is usually interpreted as describing a family $(g_t)_{t\geq 0}$ of 
conformal mappings $g_t:\Ha\setminus K_t\to \Ha$, where $(K_t)_{t\geq 0}$ is a
family of growing, bounded subsets $K_t\subset \Ha,$ also called \emph{hulls}. \\
The most important example is the Schramm-Loewner evolution SLE$(\kappa)$, which is defined via \eqref{slit0} with 
$U(t)=\sqrt{\kappa/2}B_t$, where $B_t$ is a standard Brownian motion and $\kappa\geq0$.\\

 A more general version for the growth of bounded hulls $(K_t)_{t\geq0}$ via conformal mappings\\ $g_t:\Ha\setminus K_t\to \Ha$ is given by the Loewner equation
\begin{equation}\label{slitooo}
\frac{\partial g_{t}(z)}{\partial t} =  \int_\R\frac{\nu_t(du)}{g_t(z)-u}
\quad \text{for a.e. $t\geq0$, $g_{0}(z)=z\in \Ha,$}
\end{equation}
where $(\nu_t)_{t\geq0}$ is a family of probability measures having some additional regularity properties.\\

Besides this analytic-geometric view, we might regard equation \eqref{slitooo} also as an evolution equation for a family $(\mu_t)_{t\geq 0}$ of probability measures 
on $\R$ defined via 
\[\frac1{g_t^{-1}(z)} = \int_\R \frac1{z-u} \mu_t(du).\]
This interpretation is justified by quantum probability theory: Such  families $(\mu_t)_{t\geq 0}$ arise as the
 distributions of certain quantum processes $(X_t)_{t\geq0}$ 
with monotonically independent increments. Here, a quantum process is simply a family of self-adjoint linear operators on a fixed 
Hilbert space. For the notions ``distribution of $X_t$'' and ``monotone independence'', 
we refer to Section \ref{mon_sec}.\\

For $U(t)\equiv 0$, the mappings $g_t$ from \eqref{slit0} are given as $g_t(z)=\sqrt{z^2+2t}$ and $K_t$ is the straight line segment between $0$ and $\sqrt{2t}i$. The corresponding measure 
$\mu_t$ is an arcsine distribution with mean 0 and variance $t$. In this case, the associated process $(X_t)$ is called a \emph{monotone Brownian motion}.

 We thus have the following different viewpoints on the dynamics of the Loewner equation with $U(t)\equiv 0$:

\begin{center}
 
 \begin{tabular}{|l|l|l|l|l|}
 \hline  Conformal mappings & Growing sets & Distributions $\mu_t$ & Quantum process $(X_t)$ {\color{white}$\frac{\binom{8}{9}}{\binom{8}{9}}$}\\[1mm] \hline
$g_t(z)=\sqrt{z^2+2t}$ & $K_t=[0,\sqrt{2t}i]$ & $\frac{dx}{\pi\sqrt{2t-x^2}}, x\in(-\sqrt{2t}, \sqrt{2t})$ & monotone Brownian motion {\color{white}$\frac{\binom{8}{9}}{\binom{8}{9}}$}\\ \hline

 \end{tabular}
\end{center}
 \vspace{3mm}

While the correspondence between the conformal mappings, the growing sets, and the distributions is derived 
from simple calculations, the construction of a monotone Brownian motion is rather non-trivial. 
\\

\begin{itemize}
 \item[(1)] Muraki constructed a monotone Brownian motion on a certain Fock space in \cite{MR1462227} 
 (before he introduced the notion of monotone independence around the year 2000). 
\end{itemize}

\vspace{2mm}
Just as a classical Brownian motion can be approximated by a random walk, one can construct 
a sequence of growing graphs, a ``monotone quantum random walk'', which approximates a monotone 
Brownian motion:\\

\begin{itemize}
 \item[(2)] In \cite[Theorem 5.1]{acc}, the authors construct a sequence of undirected graphs 
 $G_1, G_2, ...,$ whose adjacency matrices $A_1,A_2,...$ can be interpreted as a discrete approximation 
 of a monotone Brownian motion. The graph $G_{n-1}$ is a subgraph of $G_{n}$, and  $A_n$ can 
 be regarded as self-adjoint operator on the Hilbert space $l^2(V_n)$, where $V_n$ denotes the vertex set of $G_n$. \\
 Thus, the growing graph $(G_n)_{n\in\N}$ can be thought of as a ``monotone quantum random walk'', and the moments of 
 $A_n$ (scaled in a suitable way) converge to the moments of a monotone Brownian motion. 
\end{itemize}

\vspace{2mm}

It is natural to ask whether the constructions (1) and (2) can be extended to more general processes.
The construction of quantum processes with monotonically independent increments associated to \eqref{slitooo} has been established in the recent works \cite[Theorem 6.8]{jek17} and \cite[Theorem 1.14]{iu}. Both works regard even more general settings.\\

In this paper we are concerned with (2). O. Bauer already noted in \cite[Section A]{bauer03} that a discrete 
L\"owner evolution can be thought of as a monotone quantum random walk. 
Our main results explicitly describe theses random walks based on the construction from \cite{acc}. \\

\textbf{Outline of this work:}\\

In Section 2 we recall some facts about Loewner's differential equation and we explain its relation 
to monotone probability theory in Section 3.

In Section 4 we recall the comb product of graphs and look at certain spidernets.\\
In Section 5 we then find discrete approximations as in (2) via comb products of those spidernets for equation \eqref{slit0} with continuous non-negative driving functions (Theorem \ref{theorem10}) and for equation \eqref{slit2} with measures $\nu_t$ with $\supp \nu_t\subset [0,M]$ 
for some $M>0$ (Theorem \ref{theorem11}).


\newpage

\section{Loewner's differential equation}

\subsection{The slit Loewner equation}${}$\\[-2mm]

The slit Loewner equation is given by 
\begin{equation}\label{slit}
\frac{\partial g_t(z)}{\partial t} =  \frac{1}{g_t(z) - U(t)},  \quad
g_0(z)=z\in\Ha=\{z\in\C\,|\, \Im(z)>0\},
\end{equation} with a continuous driving function $U:[0,\infty)\to \R$. \\

The solution  yields a family $(g_t)_{t\geq0}$ of conformal mappings
$g_t:\Ha\setminus K_t\to \Ha$ with a strictly growing family $(K_t)_{t\geq0}$ of bounded sets, i.e. 
$K_s\subsetneq K_t$ whenever $0\leq s<t.$ The initial condition implies $K_0=\emptyset.$\\

Let $f_t=g_t^{-1}.$ The family $(f_t)_{t\geq 0}$ is also called a decreasing Loewner chain. From 
\eqref{slit} it follows that $(f_t)$ satisfies the following partial differential equation:
\begin{equation}\label{slit2}
\frac{\partial}{\partial t} f_{t}(z) = -\frac{\partial}{\partial z}f_{t}(z)\cdot \frac{1}{z-U(t)},
\quad f_{0}(z)=z\in \Ha.
\end{equation}

Each $f_t$ has hydrodynamic normalization. More precisely, 
\begin{equation}\label{hydro} f_t(z) = z - \frac{t}{z} + {\scriptstyle\mathcal{O}}(|z|^{-1})\end{equation}
as $|z|\to\infty$ in the sense of a non-tangential limit. 

\begin{figure}[ht]
\rule{0pt}{0pt}
\centering
\includegraphics[width=12cm]{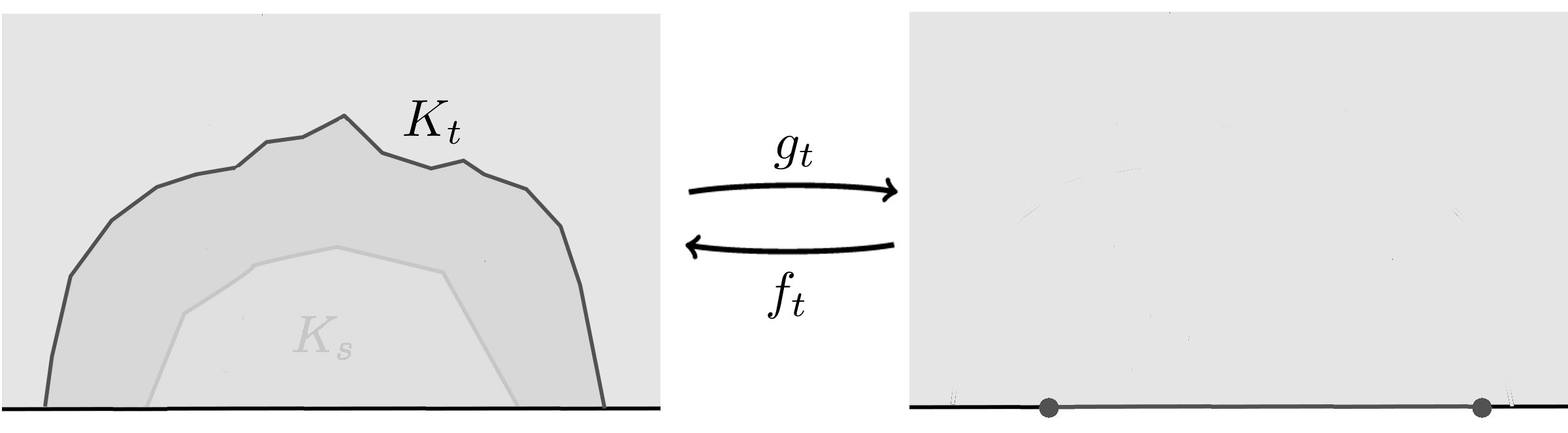}
\caption{The mappings $g_t$ and $f_t$.}
\label{fig1}
\end{figure}

\begin{example}\label{ex_1}
 For $U(t)\equiv u\in\R$, we obtain \[g_t(z)=\sqrt{(z-u)^2+2t}+u \qquad \text{and} \qquad 
    f_t(z)=\sqrt{(z-u)^2-2t}+u,   \]
 where
 the square roots are chosen 
 such that the functions map into the upper half-plane $\Ha$. We have $K_t=[u,u+\sqrt{2t}i]$, i.e. we 
 describe the growth of a straight line starting at $u$. \hfill $\bigstar$
\end{example}

\begin{remark}Assume that $K_t$ is a slit, i.e. $K_t=\gamma(0,t]$ for a simple curve $\gamma$ as in the previous 
example. Then $U$ is continuous and $g_t$ can be extended continuously to the tip $\gamma(t)$ of the slit $K_t$ and we have 
$U(t)=g_t(\gamma(t))$, see \cite[Lemma 4.2]{Lawler:2005}. \\
Not every continuous $U$ generates slits. However, if $U$ is sufficiently 
smooth, then $K_t$ is a slit, see \cite{LindMR:2010, Lind:2005, MarshallRohde:2005}.  \hfill $\bigstar$
\end{remark}
The celebrated Schramm-Loewner evolution can be defined as follows:\\
 Let $\kappa\geq 0.$ Then SLE($\kappa$) is defined as the random family $(K_t)_{t\geq 0}$ obtained by 
 \eqref{slit} with $U(t)=\sqrt{\kappa/2}B_t$, where $B_t$ is a standard Brownian motion.
 Fix some $T>0$. Then the random hull $K_T$ is a slit almost surely if and only if 
 $\kappa\in[0,4]$.\\
The corresponding random growth process $(K_t)_{t\geq0}$ was shown to be the scaling limit of
random curves from different statistical models. For SLE and the slit Loewner equation, we refer the interested reader to the book \cite{Lawler:2005}.

\subsection{A more general Loewner equation for bounded hulls}

Now we consider a more general version of equation \eqref{slit2}.

\begin{definition}
Let $(\nu_t)_{t\geq0}$ be a family of probability 
measures on $\R$ such that $t\mapsto H(t,z):=\int_\R\frac{\nu_t(du)}{z-u}$ is measurable for every $z\in\Ha$, and  assume that there exists $M>0$ such that $\supp \nu_t \subset [-M,M]$ for all $t\geq0$.
We call the function $H(t,z)$ a \emph{Herglotz vector field} and we denote the set of all 
such Herglotz vector fields by $\mathcal{H}_M$.
\end{definition}
 
\begin{definition}
A \emph{decreasing Loewner chain} on $\Ha$ is a family $(f_t)_{t\geq0}$ of univalent mappings $f_t:\Ha\to\Ha$ such that 
$f_0$ is the identity, $f_t(\Ha)\subset f_s(\Ha)$ whenever $0\leq s\leq t$, and $t\mapsto f_t$ is continuous with respect to locally uniform 
convergence.
\end{definition}

Let $H\in \mathcal{H}_M$ and consider the Loewner equation

\begin{equation}\label{slit33}
\frac{\partial}{\partial t} f_{t}(z) = -\frac{\partial}{\partial z}f_{t}(z)\cdot H(t,z)
\quad \text{for a.e. $t\geq 0$, $f_{0}(z)=z\in \Ha.$}
\end{equation}


\begin{theorem}\label{Houston} There 
exists a unique solution $(f_t)_{t\geq0}$ of equation \eqref{slit33},
which is a decreasing Loewner chain with normalization \eqref{hydro}.
Furthermore, each $f_t$ maps $\Ha$ conformally onto $\Ha\setminus K_t$ for a bounded set 
$K_t\subset {\Ha}$. \\
There exists a bound $C(t,M)>0$ such that $\sup_{z\in K_t} |z|< C(t,M)$.
\end{theorem}
\begin{proof}
The first statement follows from \cite[Theorem 4]{MR1201130},
see also \cite[Section 3]{iu}.\\

Furthermore, the condition $\supp \nu_t \subset [-M,M]$ can be used to show that there is a bound $A(t,M)>0$ 
such that every $f_t$ extends conformally onto $I(t,M):=\R\setminus[-A(t,M),A(t,M)]$ with $f_t(I(t,M))\subset \R$, see, e.g., \cite[Theorem 5.11]{jek17}.\\
This implies that there 
exists a bound $C(t,M)>0$ such that $\sup_{z\in K_t} |z|< C(t,M)$, see 
\cite[Inequality (3.14) on p.74]{Lawler:2005}. 

\end{proof}

The following convergence result is standard in Loewner theory, see e.g. \cite[Lemma 4.12]{ghkk} for a slightly different setting.

\begin{lemma}\label{aprox_lemma}
Fix $T>0$. For every $n\in\N$, let $H_n(t,z)\in \mathcal{H}_M$. Assume that there exists 
$H(t,z)\in \mathcal{H}_M$ such that  \[\int_0^t H_n(s,z)ds \to \int_0^t H(s,z)ds\]
for every $t\in[0,T]$ locally uniformly in $\Ha$ as $n\to\infty$.\\
Let $f_{n,t}$ and $f_t$ be the solutions to \eqref{slit33} for the Herglotz 
vector fields $H_n(t,z)$ and $H(t,z)$ respectively. Then 
$f_{n,t}\to f_t$ for every $t\in[0,T]$ locally uniformly in $\Ha$.
\end{lemma}
\begin{proof}It is easy to see that the set 
$\{\int_\R \frac{\nu(du)}{z-u}\,|\, \text{$\nu$ is a prob. measure with $\supp \nu\subset [-M,M]$}\}$ is a normal family. 
Thus, if $G\in \mathcal{H}_M$ and $K\subset \Ha$ is a compact set, then there exists 
$L(K)>0$ such that $|G(t,z)-G(t,w)|\leq L(K)|z-w|$ for all $z,w\in K$ and all $t\in[0,T]$.\\ 
We now look at $g_{n,t}:=f_{n,t}^{-1}$, $g_t:=f_t^{-1}$. These functions satisfy 
\eqref{slitooo} and we have
\begin{equation*}
g_{n,t}(z) =  z + \int_0^t H_n(s,g_{n,s}(z)) ds,\quad 
g_{t}(z) =  z+\int_0^t H(s,g_{s}(z)) ds.
\end{equation*}
Now let $K\subset \Ha$ be a compact set on which all $g_{n,t}$ and $g_t$ are defined. The set 
$\{g_{n,t}\,|\, t\in[0,T],n\in\N\}\cup \{g_t\,|\, t\in[0,T]\}$ is also a normal family due to 
Theorem \ref{Houston}. Hence there exists a second compact set 
$K'\subset \Ha$, $K\subset K'$, such that $g_{n,t}(z), g_{t}(z)\in K'$ for all $z\in K$, $n\in\N$, and 
$t\in[0,T]$.\\
 We know that $\int_0^t H_n(s,z)ds$ converges uniformly on $K'$ to 
$\int_0^t H(s,z)ds$ for all $t\in[0,T]$. Now fix $t\in[0,T]$. For $z\in K$ we have 
\[ |g_{n,t}(z)-g_t(z)|\leq \left| \int_0^t H_n(s,g_{n,s}(z))-H_n(s,g_{s}(z)) ds \right| + 
\left| \int_0^t H_n(s,g_{s}(z))-H(s,g_{s}(z)) ds \right| \leq \]
\[ L(K') \int_0^t |g_{n,s}(z)-g_{s}(z)| ds  + 
\eps_n, \] for a sequence $(\eps_n)_n$ converging to $0$.
Gronwall's lemma implies that $g_{n,t}\to g_t$ uniformly on $K$. Hence also $f_{n,t}\to f_t$ locally uniformly in $\Ha$.
\end{proof}

We can now prove the following result, which will reduce our problem of constructing graphs for equation \eqref{slit33} to 
equation \eqref{slit2}.

\begin{lemma}\label{Whitney}Let $H(t,z)=\int_\R\frac{\nu_t(du)}{z-u}\in \mathcal{H}_M$ and let $(f_t)_{t\geq0}$ be the corresponding 
solution to \eqref{slit33}.
 Furthermore, assume that $\supp \nu_t \subset [0,M]$ for all $t\geq0$.\\
 Fix $T>0$. Then there exists 
a sequence $U_n:[0,T]\to [0,M]$ of continuous non-negative driving functions such that the corresponding solutions $(f_{n,t})_{t\geq0}$ to \eqref{slit2} 
converge locally uniformly to $f_t$  for every $t\in[0,T]$ as $n\to\infty$.
\end{lemma}
\begin{proof}
Step 1: Assume that $H(t,z)=\frac1{z-U(t)}$ for a piecewise continuous and non-negative driving function $U$. 
Then we can clearly approximate $H(t,z)$ by a sequence 
$H_{n}(t,z)=\frac1{z-U_n(t)}$ with continuous non-negative driving functions $U_n:[0,T]\to [0,M]$
in the sense of Lemma \ref{aprox_lemma}.\\

Step 2: Next we consider the multi-slit equation, i.e. $H(t,z) = \sum_{k=1}^N\frac{\lambda_k(t)}{z-V_k(t)}$,
where $\lambda_1,...,\lambda_N:[0,T]\to[0,1]$ are continuous weight functions with
$\sum_{k=1}^N \lambda_k(t) = 1$ for all $t\in[0,T]$, and all driving functions $V_1,...,V_N:[0,T]\to[0,M]$ are continuous.\\

This Herglotz vector field can be approximated by a single-slit equation with a
piecewise continuous non-negative driving function.
We choose $m\in\N$ and divide the interval $[0,T]$ into $m$ intervals
$I_1:=[0,\frac{T}{m}], I_2:=(\frac{T}{m},\frac{T}{m}+\frac{1}{m}],...,I_m:=(T-\frac{1}{m},T]$. 
We define the driving function $U_m$ on $I_1$ as follows:
\begin{eqnarray*}
U_m(t)&=&V_1(t) \text{\quad on \quad} \left[0,T/m\cdot \lambda_1\left(T/m\right)\right],\nonumber\\
U_m(t)&=&V_2(t) \text{\quad on \quad} \left(T/m\cdot \lambda_1(T/m),T/m\cdot (\lambda_1(T/m)+\lambda_2(T/m))\right], ..., \nonumber\\
U_m(t)&=&V_N(t) \text{\quad on \quad} \left(T/m\cdot (\lambda_1(T/m)+...+\lambda_{N-1}(T/m)), T/m\right].
\end{eqnarray*}
We now repeat this construction for $I_2$,...,$I_m$. \\
Define $H_m(t,z)=\frac1{z-U_m(t)}$. Then $H_m(t,z)$ approximates 
$H(t,z)$ 
in the sense of Lemma \ref{aprox_lemma}. Together with step 1, we see that this multi-slit equation can be approximated 
by continuous non-negative driving functions.\\

Step 3:  Next we consider $H(t,z) = \sum_{k=1}^N\frac{\lambda_k(t)}{z-V_k(t)}$, where $\lambda_1,...,\lambda_N:[0,T]\to[0,1]$ are measurable weight functions with
$\sum_{k=1}^N \lambda_k(t) = 1$ for all $t\in[0,T]$, and all driving functions $V_1,...,V_N:[0,T]\to[0,M]$ are continuous.\\
For $m\in\N$, we let $H_m(t,z) = \sum_{k=1}^N\frac{\lambda_{k,m}(t)}{z-V_k(t)}$, 
where each $\lambda_{k,m}:[0,T]\to[0,1]$ is continuous, $\sum_{k=1}^N \lambda_{k,m}(t) = 1$ for all $t\in[0,T]$ and all $m\in\N$, 
and $\lambda_{k,m}\to \lambda_k$ in the $L^1$-norm as $m\to\infty$. Then $H_m(t,z)$ approximates $H(t,z)$ in the sense of Lemma \ref{aprox_lemma} as $m\to\infty$.\\

Step 4: Finally, assume that $H(t,z)=\int_\R\frac{\nu_t(du)}{z-u}\in\mathcal{H}_M$ is a general Herglotz vector field. 
 Divide $[0,M]$ into $m\in\N$ intervals: $I_{1,m}=[0,M/m], I_{2,m}=(M/m,2M/m],...,I_{m,m}=((m-1)M/m,M]$. 
For $k=1,...,m$, define $\lambda_{k,m}(t)=\nu_t(I_{k,m})$ and let $V_{k,m}(t)$ be the midpoint of $I_{k,m}$ for all $t\in[0,T]$. 
Each $\lambda_{k,m}$ is measurable, which follows from the Stieltjes-Perron inversion formula and the fact that $t\mapsto H(t,z)$ is measurable. The Herglotz vector field $H_m(t,z) = \sum_{k=1}^m\frac{\lambda_{k,m}(t)}{z-V_{k,m}(t)}$ approximates $H(t,z)$ in the sense of Lemma \ref{aprox_lemma} as $m\to\infty$.
\end{proof}

\subsection{Probabilistic interpretation of Loewner's equation}

While the geometric interpretation of Loewner's equation focuses on the growing sets 
$(K_t)_{t\geq 0}$ (or the mappings $(f_t)_{t\geq0}$), we now switch to a probabilistic
point of view, which regards a family $(\mu_t)_{t\geq 0}$ of probability measures on $\R$ instead.\\

Let $\mu$ be a probability measure on $\R$. The $F$-transform $F_\mu$ of $\mu$ is defined as the multiplicative inverse 
of the Cauchy transform of $\mu$, i.e. as the mapping 
\[F:\Ha\to \Ha, \quad F_\mu(z) := \left(\int_{\R}\frac1{z-u}\, \mu({\rm d}u)\right)^{-1}.\] 
The measure $\mu$ can be recovered from $F$ via the Stieltjes-Perron inversion formula.
We have the following simple characterization.

\begin{lemma}${}$\label{prop0}
\begin{itemize}
 \item[(a)]  A holomorphic function $F:\Ha\to\C$ is the $F$-transform of a probability measure $\mu$ on $\R$ 
 if and only if $F(\Ha)\subseteq \Ha$ and $F'(\infty)=1$ (as a nontangential derivative).\\
Furthermore, $\mu$ has mean $0$ and variance $\sigma^2$ if and only if 
\begin{equation*}
F_\mu(z) = z - \frac{\sigma^2}{z} + {\scriptstyle\mathcal{O}}(|z|^{-1})
\end{equation*}
as $|z|\to\infty$ in the sense of a non-tangential limit. 
\item[(b)] Let $\mu$, $\mu_n$, with $n\in\N$, be probability measures on $\R$. Then $\mu_n\to \mu$ with respect to 
weak convergence if and only if $F_{\mu_n}\to F_\mu$ locally uniformly on $\Ha$.
\end{itemize}
\end{lemma}
\begin{proof}
The first statement in (a) follows from the Nevanlinna representation formula and 
 \cite[Prop. 2.1]{M92} and the second statement follows from \cite[Prop. 2.2]{M92}. Statement (b) 
follows from \cite[Theorem 2.5]{M92}.
\end{proof}

We can now reformulate Theorem \ref{Houston} in the following way.
\begin{theorem}\label{Steve}Let $H\in \mathcal{H}_M$. Then there exists a unique family $(\mu_t)_{t\geq0}$ of probability measures such that $(f_t:=F_{\mu_t})_{t\geq0}$ 
solves \eqref{slit33}.
Furthermore, each $\mu_t$ has compact support, mean $0$, and variance $t$. \\
There exists a bound $C(t,M)>0$ such that $\supp \mu_t \subset [-C(t,M),C(t,M)]$.
\end{theorem}
\begin{proof}The first statement follows from combining Lemma \ref{prop0} (a) and Theorem 
\ref{Houston}, see \cite[Theorem 3.6]{monotone}.
The compactness of $\supp \mu_t$ and the existence of the uniform bound 
follow from Theorem \ref{Houston} and the Stieltjes-Perron inversion formula. 
A proof can also be found in \cite[Theorem 5.11]{jek17}.
\end{proof}

\begin{remark}
 Consider the more general Loewner equation
 \begin{equation}\label{EV_Loewner}
\frac{\partial}{\partial t} f_{t}(z) = -\frac{\partial}{\partial z}f_{t}(z)\cdot M(z,t) \quad \text{for a.e. $t\geq 0$, $f_{0}(z)=z\in \Ha,$}
\end{equation}
where,  for a.e. $t\geq 0,$ $M(\cdot, t)$ has the form
\begin{equation*}
M(z,t)=a_t + \int_\mathbb{R}\frac{1+xz}{x-z} \tau_t({\rm d}x),
\end{equation*}
 with $a_t\in\mathbb{R}$ and $\tau_t$ is a finite, non-negative Borel measure on $\mathbb{R}$. Furthermore, 
 $(z,t)\mapsto M(z,t)$ needs to satisfy certain regularity conditions.\\
Again, the solution $(f_t)$ is a family of univalent mappings $f_t:\Ha\to\Ha$ with 
$f_t(\Ha)\subseteq f_s(\Ha)$ for all $0\leq s\leq t$ and each $f_t$ is the $F$-transform 
of a probability measure on $\R.$\\
The following embedding result is proved in \cite[Theorem 1.16]{iu}: If $F_\mu$ is univalent, then there exists $T\geq 0$ and a function 
$M(z,t)$ of the above form such that the solution $(f_t)$ of \eqref{EV_Loewner} satisfies 
$f_T=F_\mu$.  \hfill $\bigstar$
 \end{remark}

 \begin{example}\label{arcs}
  The arcsine distribution $\mu_{Arc,t}$ with mean 0 and variance $t$ is given by the density 
\[\frac{dx}{\pi\sqrt{2t-x^2}}, \qquad x\in(-\sqrt{2t}, \sqrt{2t}).\] 
We have  $F_{\mu_{Arc,t}}(z)=\sqrt{z^2-2t},$ which are the mappings from Example \ref{ex_1} 
for $u=0$.\hfill $\bigstar$ 
 \end{example}

The following simple scaling relation will be useful later on.

\begin{lemma}\label{scale}Let $c,d>0$ and let $f_t=F_{\mu_t}$ be the solution to \eqref{slit2} 
with a piecewise continuous driving function $U(t)$. Consider the scaled measures $\nu_t(B)= \mu_{d\cdot t}(c\cdot B)$. Let $h_t=F_{\nu_t}$. Then $h_t$ solves 
\[\frac{\partial}{\partial t}h_{t}(z) = 
\frac{\partial}{\partial z}h_{t}(z)\cdot \frac{d/c^2}{h_{t}(z)-U(d\cdot t)/c}.\]
\end{lemma}
\begin{proof}
We have
\[h_t(z) = \left(\int_\R \frac1{z-u} \mu_{d\cdot t}(c\cdot du)\right)^{-1} = \left(\int_\R 
\frac1{z-u/c} \mu_{d\cdot t}(du)\right)^{-1} =
\left(\int_\R \frac{c}{cz-u} \mu_{d\cdot t}(du)\right)^{-1}=f_{dt}(cz)/c.\]
Then \eqref{slit2} leads to
\[\frac{\partial}{\partial t}h_{t}(z) = \frac{d}{c} \frac{\partial}{\partial t}f_{dt}(cz) =
\frac{d}{c}\frac{\partial}{\partial z}f_{dt}(cz)\cdot \frac{1}{f_{dt}(cz)-U(d\cdot t)}=
\frac{\partial}{\partial z}h_{t}(z)\cdot \frac{d/c^2}{h_{t}(z)-U(d\cdot t)/c}.\]
\end{proof}
The reason why it makes sense to consider Loewner's differential equation in this way 
is given by monotone probability theory, more precisely, by monotone increment processes.

\section{Monotone increment processes}\label{mon_sec}

Let $H$ be a Hilbert space and denote by $B(H)$ the space of all bounded linear operators on 
$H.$ In quantum probability theory, elements of $B(H)$ are regarded as non-commutative 
random variables in the following way. \\
Fix a unit vector $\xi\in H.$ Then we can define a so called \emph{state} $\Phi$ as the $\C$-linear mapping 
\[\Phi:B(H)\to\mathbb{C}, \qquad \Phi(X)=\langle\xi, X\xi\rangle.\]

Motivated by quantum mechanics, we can think of $\Phi(a)$ as
the expectation of the quantum random variable 
$a\in B(H).$ 
\begin{definition}We call $(H,\xi)$ a \emph{quantum probability space}.\\
  A self-adjoint element $a\in B(H)$ is called a \emph{quantum 
 random variable}. There exists a unique probability measure $\mu$ on $\R$ 
 such that the moments of $\mu$ are given by $\Phi(a^n)$, i.e. $\int_\R x^n \mu(dx) = \Phi(a^n)$ 
 for all $n\in\N.$ We call $\mu$ the \emph{distribution} of $a.$
 \end{definition}
 
The notion of independence is of vital importance for classical probability theory. In a certain sense, there are only five suitable notions of independence in the non-commutative setting: 
tensor, Boolean, free, monotone and anti-monotone independence;  see \cite{MR2016316}.\\
In all five cases, independence of two elements $a,b\in B(H)$ is expressed algebraically
by computation rules for mixed moments. We consider monotone independence, introduced by N. Muraki 
(\cite{MR1853184}, \cite{MR1824472}).


\begin{definition}
 Let $X_1,...,X_N\in B(H)$ be self-adjoint random variables in the quantum probability space $(H,\xi)$.
 The tuple $(X_1,X_2,...,X_N)$ is called \emph{monotonically independent} if
$$\Phi(X_{i_1}^{p_1}\dots X_{i_k}^{p_k} \dots 
X_{i_m}^{p_m})=\Phi(X_{i_k}^{p_k})\cdot 
\Phi(X_{i_1}^{p_1}\dots X_{i_{k-1}}^{p_{k-1}} X_{i_{k+1}}^{p_{k+1}} \dots 
X_{i_m}^{p_m})$$
for all $m\in\N$, $p_1,...,p_m\in \N_0$, whenever $i_{k-1}<i_k>i_{k+1}$ (one of the inequalities is eliminated when $k=1$ or $k=m$).
\end{definition}

\begin{remark}We note that sometimes, e.g. in \cite{MR1853184}, \cite{MR1824472}, a stronger condition is imposed 
in the definition of monotone independence. 
As noted in \cite[Remark 3.2 (c)]{franz07b}, both definitions coincide 
if $\xi$ is cyclic with respect to $X_1,...,X_N$.  \hfill $\bigstar$
\end{remark}

Assume that $(X,Y)$ is a pair of monotonically independent self-adjoint random variables. 
If $\alpha$ and $\beta$ are the distributions of $X$ and $Y$ respectively, then it can be shown that the distribution $\gamma$ of 
$Z=X+Y$ can be computed 
by $$F_\gamma = F_\alpha \circ F_\beta,$$
see, e.g., \cite[Theorem 3.10]{franz07b}.
This relation defines the additive monotone convolution 
$\alpha \rhd \beta := \gamma.$

\begin{remark}[Literature]
For quantum probability theory (including its important relations to random matrices), we refer the reader to introductions such as 
\cite{Att, DNV92, meyer, musp}.\\
The five notions lead to central limit theorems, the investigation of quantum stochastic processes with independent increments, and to quantum stochastic differential equations. 
The latter topics are treated in detail in the books \cite{MR2132092, MR2213451}.\\ 
Finally, we also refer to \cite{Oba17}, where the author shows how quantum probability theory can be applied to the spectral analysis of graphs.
The different notions of independence appear in connection with certain products for graphs. \hfill $\bigstar$
\end{remark}

We now explain the relation of monotone independence to the Loewner equation. Let $(f_t)_{t\geq0}$ be the solution to \eqref{slit2} 
and let $0\leq s\leq t$. 
Then $f_t=f_s\circ f_{s,t}$ for some univalent 
function $f_{s,t}:\Ha\to\Ha$, as the image domains $f_t(\Ha)$ are decreasing.\\
As $f_0$ is the identity, we have $f_t=f_{0,t}$. 
We can apply Lemma \ref{prop0} to see that we can write $f_{s,t}=F_{\mu_{s,t}}$ for a probability measure $\mu_{s,t}$ 
on $\R$. Hence, we have \begin{equation}\label{triv}
       \mu_{0,t} = \mu_{0,s}\rhd \mu_{s,t},                  
                        \end{equation}
which suggests that there might be an 
underlying family $(X_t)_{t\geq0}$ of self-adjoint operators such that $X_0=0$, $X_s$ and $X_t-X_s$ 
are independent for $s\leq t$, and $\mu_{s,t}$ is the distribution of $X_t-X_s.$ Equation \eqref{triv} would then follow from 
\[X_t = X_s + (X_t-X_s).\] This leads us to the following definition.

\begin{definition}\label{def_saip}
Let $(H,\xi)$ be a quantum probability space and $(X_t)_{t\ge 0}$ a family of bounded self-adjoint operators on
$H$ with  $X_0=0$.  We 
call $(X_t)$ a \emph{self-adjoint operator-valued additive
monotone increment process (SAMIP)} if the following conditions are satisfied:
\begin{itemize}
	\item[(a)] For every $s\geq 0,$ the mapping $t\mapsto \mu_{s,t}$ is continuous w.r.t.\ weak convergence, where $\mu_{s,t}$ denotes the distribution of 
	$X_t-X_s$.
	\item[(b)]The tuples \[
(X_{t_1},X_{t_2}-X_{t_1},\ldots,X_{t_n}-X_{t_{n-1}})
\]
are monotonically independent for all $n\in\mathbb{N}$ and all $t_1,\ldots,t_n\in\mathbb{R}$ s.t.\ $0\le t_1\le t_2\le\cdots\le t_n$. 
\end{itemize}

\end{definition}
We also write $\mu_t$ instead of $\mu_{0,t}$ for the distribution of $X_t$.

\begin{example}[Monotone Brownian motion] \label{ex_m_b_m}
Recall the arcsine distribution $\mu_{Arc,t}$ with mean 0 and variance $t$ from Example \ref{arcs}.
 The normalized distribution $\mu_{Arc,1}$ is the monotone analogue of the normal distribution
 from classical probability, as it is the limit distribution in the central limit theorem of monotone probability theory,
 see \cite[Theorem 2]{MR1853184}. \\
A SAMIP $(X_t)$ with distributions $\mu_t=\mu_{Arc,t}$ is thus called a \emph{monotone Brownian motion}. 
We have  $F_{\mu_{Arc,t}}(z)=\sqrt{z^2-2t}.$ These mappings 
simply describe the growth of a straight line starting at $0$, see Example \ref{ex_1}. 
In \cite{MR1462227}, Muraki constructed a monotone Brownian motion on a certain Fock space.  \hfill $\bigstar$
\end{example}

The following result follows from \cite[Theorem 6.8]{jek17} or \cite[Theorem 1.14]{iu}.

\begin{theorem}\label{khl}Let $H\in\mathcal{H}_M$ and let $(f_t)_{t\geq0}$ be the solution to \eqref{slit33}. Write 
$f_t=f_s\circ f_{s,t}$ and define $\mu_{s,t}$ by 
$f_{s,t}=F_{\mu_{s,t}}.$ Then there exists a SAMIP $(X_t)_{t\geq 0}$ on a quantum probability space $(H,\xi)$ such that the distribution of 
$X_t-X_s$ is given by $\mu_{s,t}.$ 
\end{theorem}

%
\newpage
\section{Spidernets and comb products}

We now follow the work \cite{acc} and modify its main result (Theorem 5.1), which can be interpreted as a discrete approximation 
of a monotone Brownian motion, a ``monotone quantum random walk'', via adjacency matrices 
of certain graphs.\\

Let $V$ be a vertex set, finite or countable infinite, with a distinguished vertex $o\in V$.\\
Let $A:V\times V\to \{0,1\}$ be a symmetric matrix with $A_{xx}=0$ for all $x\in V$. \\

We can interpret $A$ as the adjacency matrix of an undirected (loop-free) graph with vertex set $V$,
where $A_{xy}=1$ if and only if $x\sim y$, i.e. $x$ and $y$ are connected by an edge.

\begin{definition}
We define a \emph{(rooted)} \emph{graph} as such a triple $G=(V,A,o)$.\\
 For $x\in V$, the degree $\deg(x)$ of $x$ is defined as $\sum_{y\in V} A_{xy}$. The degree of the graph is defined as $\deg(G):=\deg(A):=\sup_{x\in V}\deg(x)$.
\end{definition}


If $\deg(A)<\infty$, then $A$ can be regarded as a bounded self-adjoint operator 
on the Hilbert space $l^2(V)$, see \cite[Theorem 3.1]{mw89}.
The distinguished vertex $o\in V$ enables us to regard $A$ as a quantum 
random variable on the quantum probability space 
$(l^2(V), \delta_o)$, where $\delta_o\in l^2(V)$ with $(\delta_o)(o)=1, (\delta_o)(x)=0$
for $x\not=o$. 

\begin{example}\label{zet}
 Let $V=\Z$ with $A_{jk}=1$ if and only if $|j-k|=1$ and $0$ otherwise. Choose $o=0.$ Then 
 the distribution of $A$ within the probability space $(l^2(\Z), \delta_0)$ is given by the arcsine 
 distribution with mean 0 and variance 2, see \cite[Section 6.1]{acc}. \hfill $\bigstar$
\end{example}


Let $G_1=(V_1, A^1, o_1), G_2=(V_2, A^2, o_2)$ be two graphs. Then the comb product 
$G_1 \rhd G_2 = (V_3,A^3,o_3)$ (with respect to $o_2$) 
is defined as the graph with
vertices $V_3=V_1\times V_2$, distinguished vertex 
$o_3=(o_1,o_2)$, and 
\begin{equation}\label{comb_p} A^3_{(xx')(yy')} = A^1_{xx'}\delta_{yo_2}\delta_{y'o_2} + \delta_{xx'}A^2_{yy'}. \end{equation}
Here we use the symbol $\delta_{xy}=1$ if $x=y$, $\delta_{xy}=0$ if $x\not=y$.
It can be verified that $(x,y)\sim (x',y')$ if and only if 
\begin{itemize}
 \item $x\sim x'$, $x\not= x'$ and $y=y'=o_2$, or
 \item  $x= x'$, $y=y'=o_2$, and $x\sim x$ or $o_2\sim o_2$, or
 \item $x=x'$ and $y\sim y',$ $(y,y')\not=(o_2,o_2)$.
\end{itemize}

 \begin{figure}[ht]
\rule{0pt}{0pt}
\centering
\includegraphics[width=8cm]{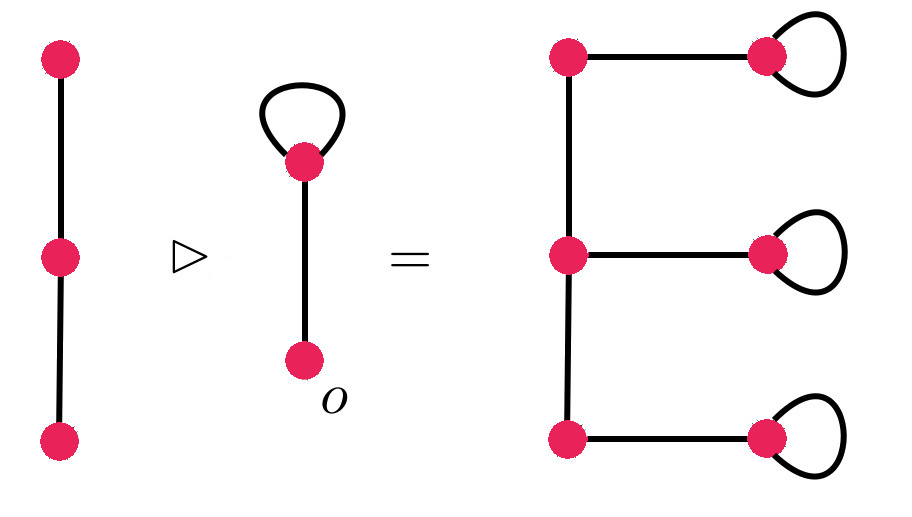}
\caption{The comb product of two graphs.}
\end{figure}

If $\deg(G_1), \deg(G_2)<\infty$, then the adjacency matrix $A^3$ of $G_1 \rhd G_2$ acts on $l^2(V_1\times V_2)\simeq l^2(V_1) \otimes l^2(V_2).$ 



The following lemma is a slightly more general version of \cite[Theorem 3.1]{acc}. Its proof follows 
 from definition \eqref{comb_p} and by induction.  

\begin{lemma}\label{dreidrei}
Let $G_1=(V_1,A^1,o_1),...,G_n=(V_n,A^n,o_n)$ be graphs. Denote by 
$I^k$ the identity on $l^2(V_k)$ and by $P^k$ the projection from $l^2(V_k)$ onto the subspace spanned by 
$\delta_{o_k}$, i.e. $(P^k(\psi))(y) = \delta_{yo_k}\psi(o_k).$  Denote by $B$ the adjacency matrix of the graph 
$G_1 \rhd G_2 \rhd ... \rhd G_n$. 
Then \begin{equation}\label{sumii}
 B= \sum_{j=1}^n I^1 \otimes ... \otimes I^{j-1} \otimes A^j \otimes P^{j+1} \otimes ... \otimes P^n.       
        \end{equation}
\end{lemma}

Assume that $\sup\{\deg(v)\,|\, v\in V_j\}<\infty$ for all $j=1,...,n$. Then 
the adjacency matrix $B$ can be regarded as a quantum random variable in 
$(l^2(V_1\times ... \times V_n), \delta_{o_1}\otimes ... \otimes \delta_{o_n}).$ 
By \cite[Proposition 4.1]{acc}, the random variables 
$(I^1 \otimes ... \otimes I^{j-1} 
\otimes A^j \otimes P^{j+1} \otimes ... \otimes P^n)_{j\in(1,...,n)}$ are monotonically independent. 
Thus the distribution of $B$ is given by the monotone convolution of the distributions of 
the summands in \eqref{sumii}.  Furthermore, it is easy to see that 
the moments of 
$I^1 \otimes ... \otimes I^{j-1} \otimes A^j \otimes P^{j+1} \otimes ... \otimes P^n$ with respect to 
$(l^2(V_1\times ... \times V_n), \delta_{o_1}\otimes ... \otimes \delta_{o_n})$ agree with the moments 
of $A^j$ within $(l^2(V_j), \delta_{o_j}).$ Thus we obtain:

\begin{lemma}\label{viervier}Assume that $\sup\{\deg(v)\,|\, v\in V_j\}<\infty$ for all $j=1,...,n$. 
	Then the random variables 
$(I^1 \otimes ... \otimes I^{j-1} 
\otimes A^j \otimes P^{j+1} \otimes ... \otimes P^n)_{j\in(1,...,n)}$ are monotonically independent in the quantum probability space 
$(l^2(V_1\times ... \times V_n), \delta_{o_1}\otimes ... \otimes \delta_{o_n})$. Let $\mu_j$ be the distribution of $A_j$ within $(l^2(V_j), \delta_{o_j})$. Then $B$ has the
distribution 
\[ \mu_1 \rhd \mu_2 \rhd ... \rhd \mu_n.\]
\end{lemma}

We now construct special graphs whose distributions will be related to the Loewner equation.\\
We denote by $d(x,y)$ the length of the shortest walk 
within  a graph connecting $x$ and $y$. For $\eps\in\{-1,0,+1\}$, we define 
for any $x\in V$, 
\[\omega_{\eps}(x)=|\{y\in V\,|\, y\sim x, d(o,y)=d(o,x)+\eps\}|.\]

Let $a\in\N, b\in\N\setminus\{1\}$ and $c\in \N$ with $c \leq b-1$.
A \emph{spidernet with data $(a,b,c)$}, see \cite[Def. 4.25]{hora}), is a graph $(V,A,o)$ with root $o\in V$
such that 
\[ \omega_{+1}(o)=a,\quad \omega_{-1}(o)=\omega_0(o)=0, \quad \text{and} \quad
 \omega_{+1}(x)=c,\quad \omega_{-1}(x)=1,\quad \omega_{0}(x)=b-1-c\]
for all $x\in V\setminus\{o\}$ (and $A_{xy}\in\{0,1\}$ for all $x,y\in V$). 

  \begin{figure}[ht]
\rule{0pt}{0pt}
\centering
\includegraphics[width=8cm]{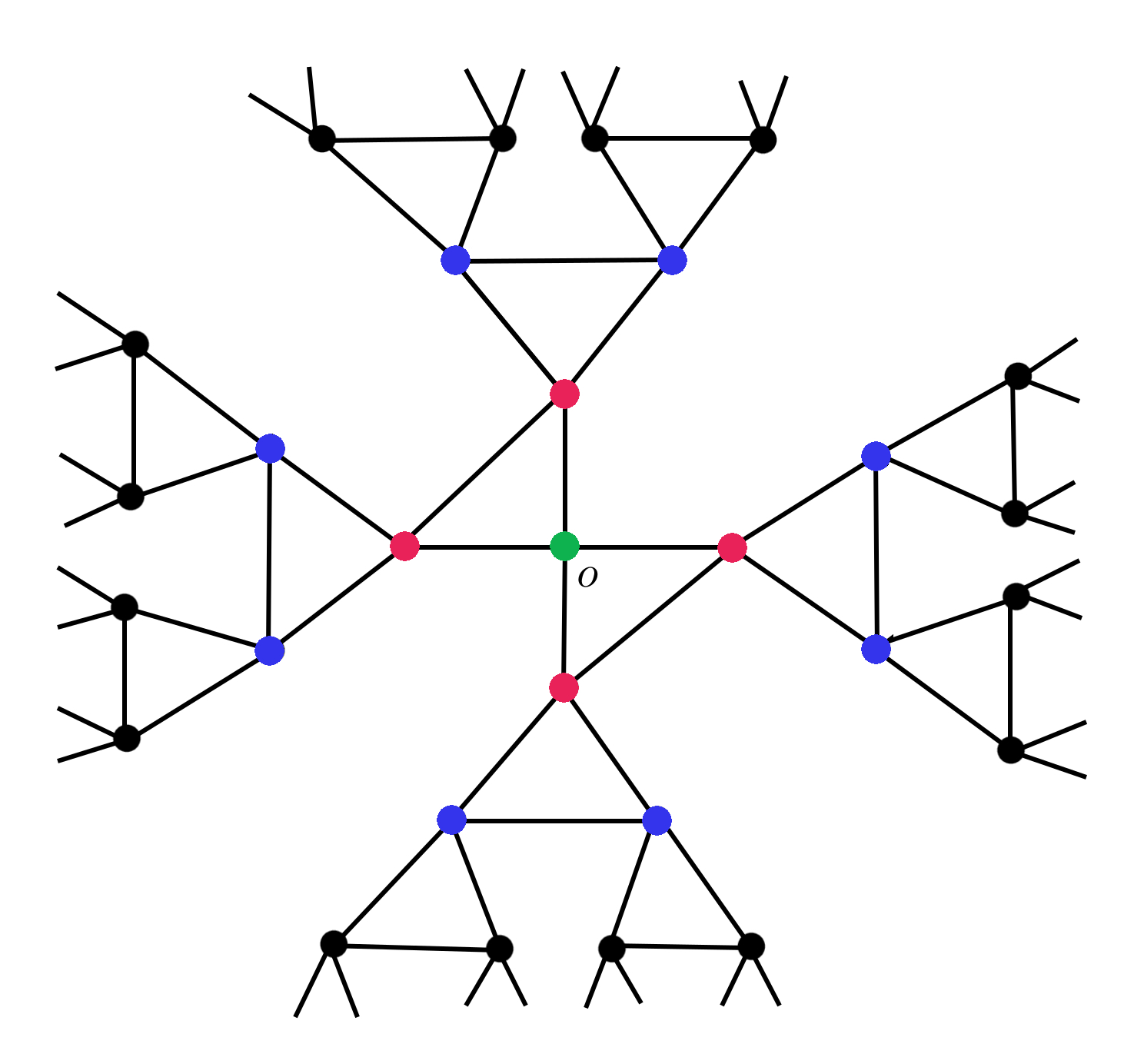}
\includegraphics[width=8cm]{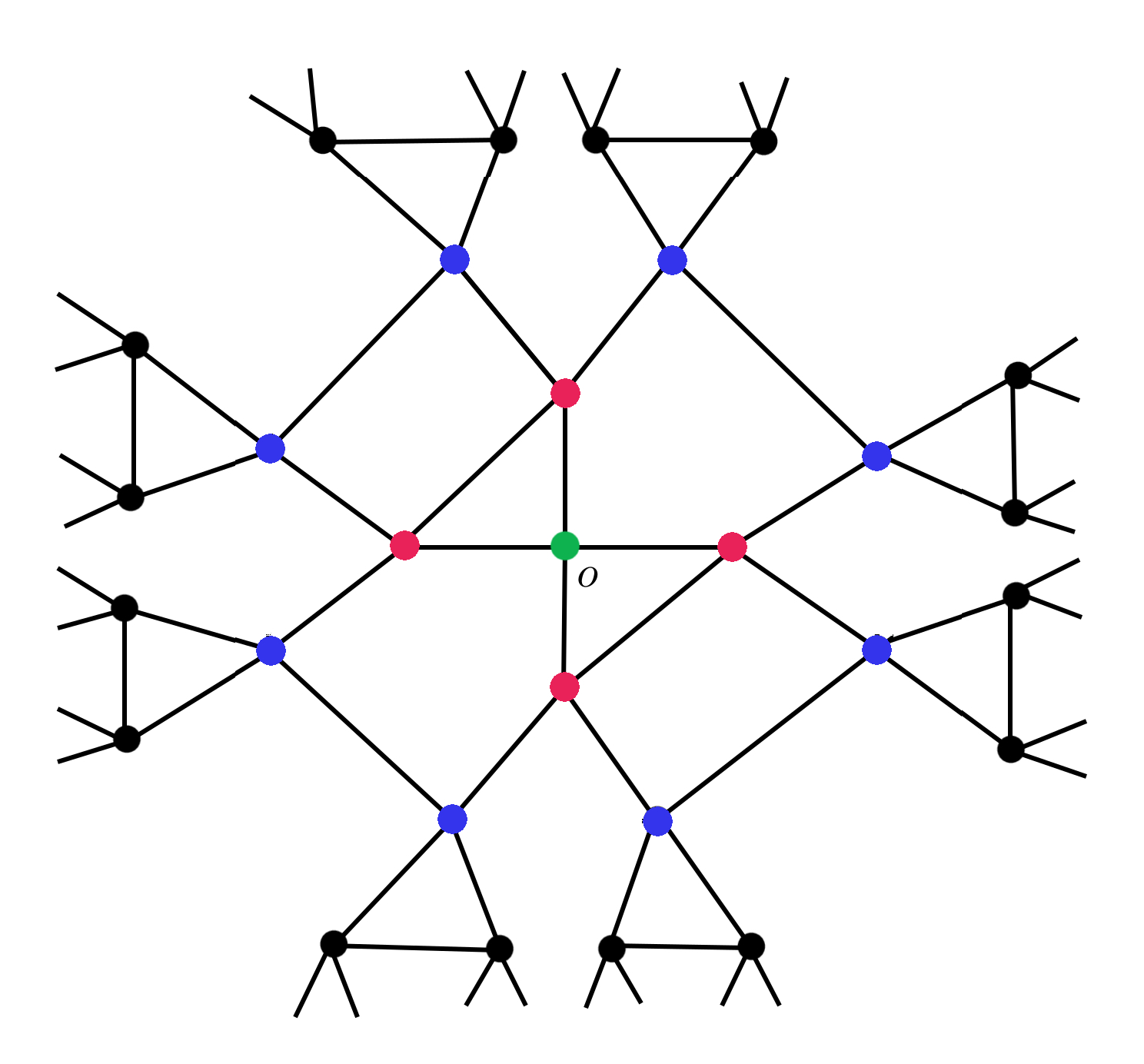}
\caption{Two spidernets with data $(4,4,2).$}
\end{figure}

\begin{lemma}[See Thm. 4.29 in \cite{hora}.]\label{lemma0}
The spectrum of the adjacency matrix of 
a spidernet w.r.t. the quantum probability space $(l^2(V), \delta_o)$ is the free Meixner law 
$m_{a,c,b-1-c}$.
\end{lemma}

The free Meixner law is described in \cite[Section 4.5]{hora}. We will only need the following property.

\begin{lemma}\label{lemma000}
Let $n\in\N, u\in\N_0$. Then the distribution $m_{2n, n, u}$ has
 $F$-transform $\sqrt{(z-u)^2-4n}+u$. It has $0$ mean and variance $2n$.
\end{lemma}
\begin{proof}
This can be easily verified by using the explicit formula \cite[Equation (B.1)]{io}. 
\end{proof}

We now combine the following two observations:
\begin{itemize}
 \item[(A)]  On the one hand, by Lemma \ref{lemma0}, $m_{2n, n, u}$  is the distribution of a
 spidernet with data $(2n, n+1+u, n)$; provided such a spidernet exists.\\
 From looking at the $2n$ vertices with $d(o,x)=1$, we get the necessary condition 
 $b-1-c = u \leq 2n-1$. Conversely, one can verify that for each $n\in\N$ and every
 $u\in\{0,...,2n-1\}$ 
 there exists a spidernet with data $(2n, n+1+u, n)$. We denote by $S_{n,u}$ 
 a fixed spidernet with such data.
 \item[(B)] On the other hand, we obtain $F_{m_{2n,n,u}}(z)=\sqrt{(z-u)^2-4n}+u$ as the solution of the
 Loewner equation with $U(t)\equiv u$ at $t=2n$, see Example \ref{ex_1}. 
 Obviously, we can also write  $m_{2n,n,u}=\delta_{-u} \rhd \mu_{Arc,2n}\rhd \delta_{u}$, see Example \ref{ex_m_b_m}.
\end{itemize}

\vspace{2mm}
Hence, approximating a driving function by piecewise constant driving functions is related to 
approximating the corresponding measures by distributions of spidernets.

%
 
 \begin{figure}[ht]
\rule{0pt}{0pt}
\centering
\includegraphics[width=7.5cm]{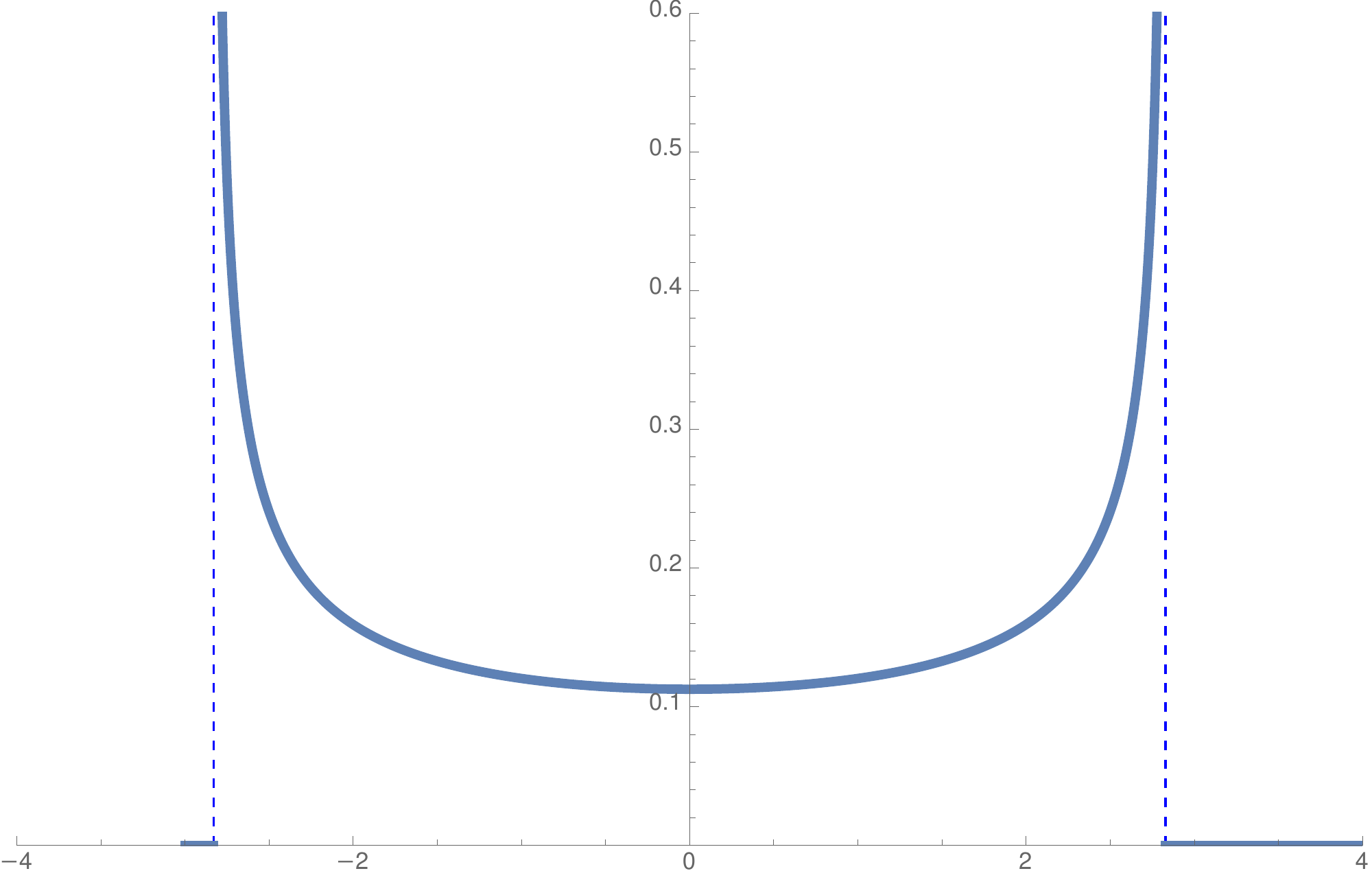}
\includegraphics[width=7.5cm]{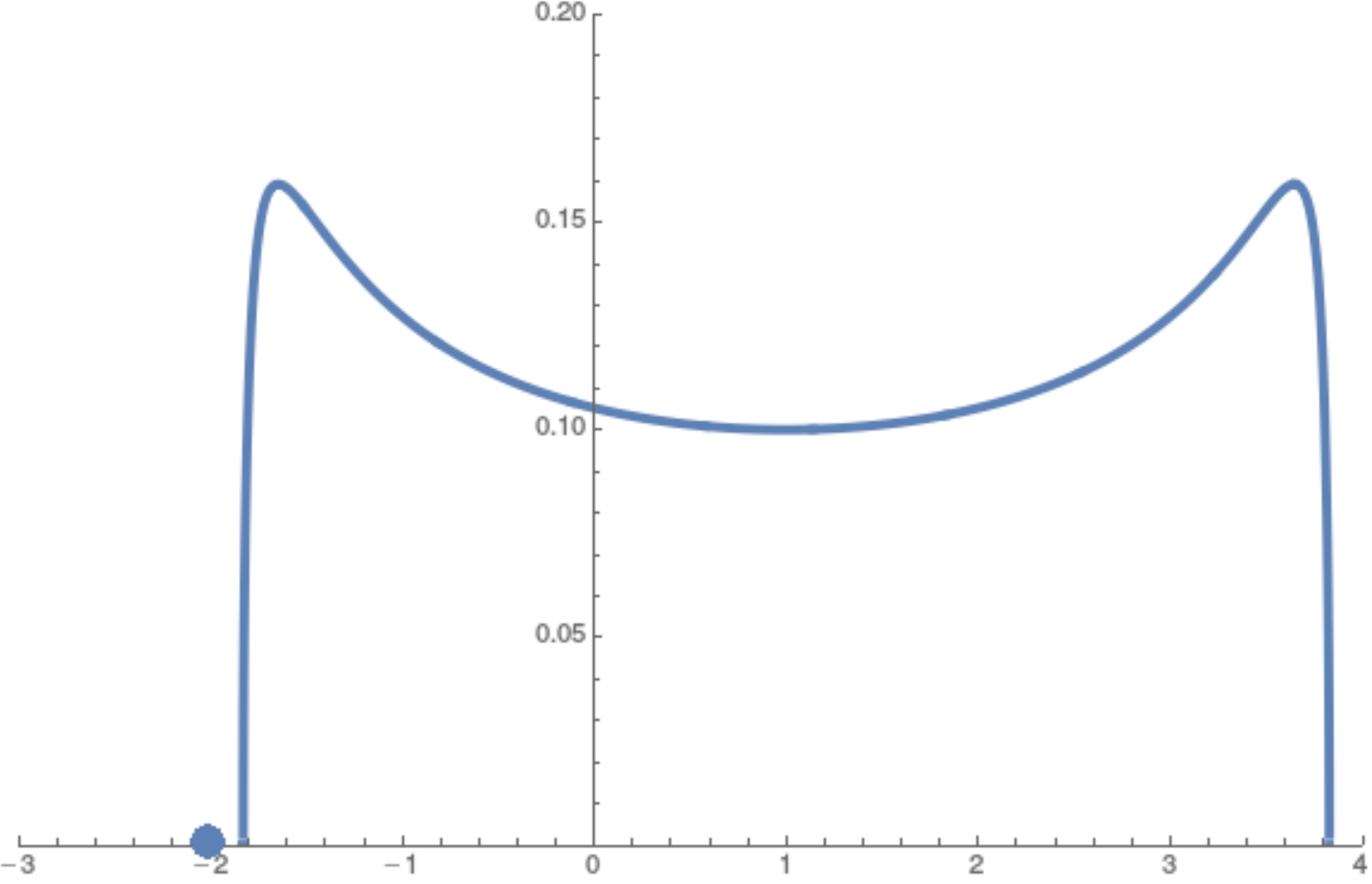}
\caption{Left: The free Meixner law $m_{4,2,0}$ is simply the arcsine distribution. Right: The density of $m_{4,2,1}$ in $[1-2\sqrt{2},1+2\sqrt{2}]$ and its atom at $-2$.}
\end{figure}

\newpage
\section{Approximation via spidernets}

\subsection{Slit equation with continuous non-negative driving functions}\label{mons0}${}$\\[-2mm]

We now consider a driving function $U:[0,\infty)\to\R$ which is continuous and non-negative. \\
Let $(f_t)_{t\geq0}$ be the solution to \eqref{slit2}  and denote by $(\mu_t)_{t\geq0}$ 
the probability measures with $F_{\mu_t}=f_t$. Furthermore, let $(X_t)_{t\geq 0}$ be a 
corresponding SAMIP process given by Theorem \ref{khl}.\\

Fix some $T>0$. We would like to approximate $(X_t)_{t\in[0,T]}$ by a discrete quantum process, where each random variable is 
the adjacency matrix of a graph.
By means of the lemmas above, we can now proceed as follows.\\

Choose $n_0\in\N$  such that   
\begin{equation}\label{est0}0 \leq U(t) \leq \sqrt{\frac{T}{2}}\left(2\sqrt{n}-\frac1{\sqrt{n^3}}\right) \quad \text{on $[0,T]$}
\end{equation}
for all $n\geq n_0$.\\

Now assume that $n\geq n_0$. For $k=1,...,n$, we define
\[u_{n,k}= 
\lfloor \sqrt{2T}\sqrt{n}\cdot \frac{U(k/n\cdot T)}{\frac{T}{n}}\rfloor \in 
\{0, ..., 2n^2-1\}.\]
Here, $\lfloor x \rfloor$ denotes the largest $m\in \N_0$ with $m\leq x$. Note that 
\eqref{est0} implies that the spidernet $S_{n^2,u_{n,k}}$ exists for all $k=1,...,n$. 
We denote by $V_{n,k}$ the vertex set and by $o_{n,k}$ the root of $S_{n^2,u_{n,k}}$.

\begin{theorem}\label{theorem10}

For $k=1,...,n$, let $\mathcal{C}_{n,k}$ be the graph
\[ \mathcal{C}_{n,k} := S_{n^2, u_{n,1}} \rhd  S_{n^2, u_{n,2}}  \rhd ... \rhd 
S_{n^2, u_{n,k}}.\] 
 Then $(\mathcal{C}_{n,k})_{k=1,...,n}$ is a an approximation 
of the quantum process $(X_t)_{t\in[0,T]}$ in the following sense:
\begin{itemize}
\item[(a)]Let $A_{n,k}$ be the adjacency matrix of $\mathcal{C}_{n,k}$. Denote by $\mu_{n,k}$
the distribution of $A_{n,k}$ with respect to the 
quantum probability space $(l^2(V_{n,1}\times ... \times V_{n,k}),\delta_{o_{n,1}} \otimes ... \otimes \delta_{o_{n,k}})$. Then 
\[ \lim_{n\to\infty}
 \mu_{n,\lfloor tn/T \rfloor}(\sqrt{2n^3/T}\; \cdot ) = \mu_t(\cdot)
\]
 with respect to  weak convergence for all $t\in[0,T]$. The limit also holds true with respect to 
 the convergence of all moments.
 \item[(b)] 
 Consider 
the quantum probability space $(l^2(V_{n,1}\times ... \times V_{n,n}),\delta_{o_{n,1}} \otimes ... \otimes \delta_{o_{n,n}})$.
Extend $A_{n,k}$ to $l^2(V_{n,1}\times ... \times V_{n,n})$ by $\mathcal{A}_{n,k}:=A_{n,k}\otimes P^{n,k+1}\otimes ... 
\otimes P^{n,n},$ where $P^{n,j}$ denotes the projection in $l^2(V_{n,j})$ onto 
$\delta_{o_{n,j}}$. Then the increments
$(\mathcal{A}_{n,1}, \mathcal{A}_{n,2}-\mathcal{A}_{n,1}, ..., \mathcal{A}_{n,n}-\mathcal{A}_{n,n-1})$ are monotonically independent.

\end{itemize}
\end{theorem}

\begin{remark}Note that the graph that corresponds to $\mathcal{A}_{n,k}$ is simply an embedding of $\mathcal{C}_{n,k}$ within 
a larger vertex set.   \hfill $\bigstar$
\end{remark}

\begin{proof}Statement (b) follows directly from Lemmas \ref{dreidrei} and \ref{viervier}.\\

Let $U_n:[0,2n^3]\to\R$ be the function which is constant $u_{n,1}$ on $[0,2n^2]$, 
constant $u_{n,2}$ on $(2n^2,4n^2]$, etc.\\
 Let $f_{n,t}$ be the 
solution to \eqref{slit2} with this driving function and define the measures $\alpha_{n,t}$ by 
$F_{\alpha_{n,t}}=f_{n,t}$. 
By Example \ref{ex_1} and Lemma \ref{lemma000} we have 
\[\alpha_{n,2n^2} = m_{2n^2,n^2,u_{n,1}}.\]
Starting the Loewner equation \eqref{slit2} for $h_t$ at $t=2n^2$ with initial value $h_{2n^2}(z)=z$
and driving function $U_n(t)$ yields the mappings $(h_t)$ that satisfy 
$f_{n,t} = f_{n,2n^2} \circ h_{t}.$ Obviously, $h_{4n^2}=F_{m_{2n^2,n^2,u_{n,2}}}$ and thus 
$\alpha_{n,4n^2} =m_{2n^2,n^2,u_{n,1}}\rhd m_{2n^2,n^2,u_{n,2}}.$ By induction we obtain 
\begin{equation*} \alpha_{n,2kn^2} = \rhd_{j=1}^{k} m_{2n^2,n^2,u_{n,j}}.
\end{equation*}
On the other hand, Lemmas \ref{dreidrei}, \ref{viervier}, \ref{lemma0} imply
\begin{equation}\label{uu00}\mu_{n,k}=\rhd_{j=1}^{k} m_{2n^2,n^2,u_{n,j}}
\end{equation}
for all $k=1,...,n$.

The function $V_n:[0,T]\to\R, V_n(t):=\sqrt{\frac{T}{2n^3}} \cdot U_n(t/T\cdot 2n^3)$ is constant 
 on the intervals $(\frac{(k-1)T}{n}, \frac{kT}{n}],$ $k=1,...,n$. We have 

\begin{eqnarray*}
&& U(k/n\cdot T)- V_n(k/n\cdot T) = 
U(k/n\cdot T)- \sqrt{\frac{T}{2n^3}} \cdot U_n(k\cdot 2n^2)=\nonumber\\
&&U(k/n\cdot T)- \sqrt{\frac{T}{2n^3}} \cdot \lfloor \sqrt{2T}\sqrt{n}\cdot \frac{U(k/n\cdot T)}{\frac{T}{n}}\rfloor \leq 
\sqrt{\frac{T}{2n^3}}.
\end{eqnarray*}

Now let $t\in(\frac{(k-1)T}{n}, \frac{kT}{n})$ and denote by $\omega:[0,T]\to[0,\infty)$ a modulus of continuity of $U$ for $[0,T]$, i.e. 
$|U(x)-U(y)|\leq \omega(|x-y|)$ for all $x,y\in[0,T]$, and $\omega$ is increasing, vanishes at $0$, and is continuous at $0$. We have
\begin{eqnarray*}
&&|U(t)-V_n(t)|=|U(t)-V_n(kT/n)|\leq \nonumber \\
&&|U(t)-U(kT/n)| + |U(kT/n)-V_n(kT/n)|\leq \omega\left(\frac{T}{n}\right) + \sqrt{\frac{T}{2n^3}}.
\end{eqnarray*}
Finally, for $t=0$ we have $V_n(0)=V_n(T/n)$ and thus
\[|U(0)-V_n(0)|=|U(0)-U(T/n)| + |U(T/n)-V_n(T/n)|\leq \omega\left(\frac{T}{n}\right) + \sqrt{\frac{T}{2n^3}}.
\]
Hence,  we obtain
\begin{eqnarray}\label{conv0}
\sup_{t\in[0,T]}|U(t)-V_n(t)|\to 0 \quad \text{as $n\to\infty$}.
\end{eqnarray}

Let $(h_{n,t})_{t\in [0,T]}$ be the Loewner chain that corresponds to $V_n$. Define the measures 
$\nu_{n,t}$ by $h_{n,t}=F_{\nu_{n,t}}.$ 
Note that $V_n$ has the form $V_n = U_n(d \cdot t)/c$ with $d=c^2$. Hence, by Lemma \ref{scale} we have 
\[\nu_{n,t}(M) = \alpha_{n,t/T\cdot 2n^3}(\sqrt{2n^3/T}\cdot M)\]
for all $t\geq0$ and all Borel subsets $M\subset \R$. If $t$ has the form $t=kT/n, k=1,...,n,$ then
\eqref{uu00} gives
\begin{eqnarray*}\nu_{n,t}(M) &=& \mu_{n,k}(\sqrt{2n^3/T}\cdot M)  
= (\rhd_{j=1}^{k} m_{2n^2,n^2,u_{n,j}} )(\sqrt{2n^3/T}\cdot M) \nonumber\\ 
&=& 
 (\rhd_{j=1}^{tn/T} m_{2n^2,n^2,u_{n,j}})(\sqrt{2n^3/T}\cdot M).\end{eqnarray*}
For every $t\in[0,T]$ we have $h_{n,t} \to f_t$ locally uniformly because of
\eqref{conv0} and Lemma \ref{aprox_lemma}. 
By Lemma \ref{prop0} (b) we have $\nu_{n,t} \to \mu_t$ with respect to weak convergence, or 
\[ \mu_{n,\lfloor tn/T \rfloor}(\sqrt{2n^3/T}\cdot ) = (\rhd_{j=1}^{\lfloor tn/T \rfloor} m_{2n^2,n^2,u_{n,j}})(\sqrt{2n^3/T}\, \cdot )  \to \mu_t(\cdot).\] 
It remains to show that this limit also holds 
with respect to convergence of all moments.\\
As there is a uniform bound for the family $(V_n)_{n}$ on $[0,T]$, Theorem \ref{Steve} implies that there exists 
$C(t)>0$ such that $\supp \nu_{n,t}\subset[-C(t),C(t)]$ for all $n$ and all $t\in[0,T]$. Thus, weak convergence of $\nu_{n,t}$ 
is equivalent to convergence of all its moments. 
\end{proof}

\subsection{General Loewner equation}\label{mons1}${}$\\[-2mm]

Consider equation \eqref{slit33} with the additional condition that $\supp \nu_t \subset [0,M]$ for all $t\geq0$.
 Let $(f_t)_{t\geq0}$ be the solution to the corresponding Loewner equation and denote by $(\mu_t)_{t\geq0}$ the probability measures with $F_{\mu_t}=f_t$. Furthermore, let $(X_t)_{t\geq 0}$ be a 
corresponding SAMIP process given by Theorem \ref{khl}. The process $(X_t)_{t\in[0,T]}$ can be approximated by graphs in the following way.


\begin{theorem}\label{theorem11}
Choose $n_0\in\N$ such that $M\leq\sqrt{\frac{T}{2}}\left(2\sqrt{n}-\frac1{\sqrt{n^3}}\right)$ for all  $n\geq n_0$.
There exists a family $(\mathcal{C}_{n,k})_{n\geq n_0, k=1,...,n}$ of rooted graphs such that:
\begin{itemize}
 \item[(a)]  For each $n\geq n_0$,  $(\mathcal{C}_{n,k})_{k=1,...,n}$ can be considered as  graphs with common vertex set $V_n$ and common root $o_n$. 
Let $A_{n,k}$ be the adjacency matrix of $\mathcal{C}_{n,k}$. Then the increments
$(A_{n,1}, A_{n,2}-A_{n,1}, ..., A_{n,n}-A_{n,n-1})$ are monotonically independent with respect to the quantum probability space
 $(l^2(V_n),\delta_{o_n})$.
\item[(b)] Denote by $\mu_{n,k}$ the distribution of $A_{n,k}$. Then 
\[ \lim_{n\to\infty}
 \mu_{n,\lfloor tn/T \rfloor}(\sqrt{2n^3/T}\; \cdot ) = \mu_t(\cdot)
\]
 with respect to  weak convergence for all $t\in[0,T]$. The limit also holds true with respect to 
 the convergence of all moments.
\end{itemize}
\end{theorem}

\begin{remark}\label{Joachim} By Theorem \ref{Steve} we know that $\int_\R x \mu_t(dx) = 0$ and $\int_\R x^2 \mu_t(dx) = t$ for all $t\geq0$. 
Theorem \ref{theorem11} implies that $\int_\R x^k \mu_t(dx) \geq 0$ for all $k\geq 3$, as the distributions $\mu_{n,k}$ of the adjacency matrices (whose entries are either $0$ or $1$) obviously have non-negative moments.    \hfill $\bigstar$  
\end{remark}

\begin{proof}
Due to Lemma \ref{Whitney} there exists a sequence of continuous non-negative driving functions 
$U_m:[0,T]\to[0,M]$ such that the corresponding solution $f_{m,t}$ to 
\eqref{slit2} converges locally uniformly to $f_t$ for all 
$t\geq0$ as $m\to\infty$. Write $f_{m,t}=F_{\mu_{m,t}}$. Then Lemma \ref{prop0} (b) implies that $\lim_{m\to\infty}\mu_{m,t}=\mu_t$.\\

Let $\mathcal{C}_{n,k;m}$ be the graphs from Theorem \ref{theorem10} for the driving function 
$U_m$ with distributions $\mu_{n,k;m}$. Note that $n\geq n_0$ and \eqref{est0} together with the bound $U_m(t)\leq M$ imply that $n$ is large enough 
to construct these graphs. Then 
\[  \lim_{n\to\infty} \mu_{n,\lfloor tn/T \rfloor;m}(\sqrt{2n^3/T}\; \cdot ) = \mu_{m,t}(\cdot). \]

A diagonalization argument (note that there is a metric for probability measures on $\R$
which is compatible with weak convergence, e.g.\ the L\'evy-Prokhorov distance) gives us a sequence $m(n)$ converging to $\infty$
  such that
\[  \lim_{n\to\infty} \mu_{n,\lfloor tn/T \rfloor;m(n)}(\sqrt{2n^3/T}\; \cdot ) = \mu_{t}(\cdot). \]

Hence, the graphs $\mathcal{C}_{n,k}:=\mathcal{C}_{n,k;m(n)}$ (where $\mathcal{C}_{n,k}$ is regarded as a 
subgraph of $\mathcal{C}_{n,n}$) satisfy all required conditions.

\end{proof}


\newpage
\def\cprime{$'$}

\end{document}